\documentclass [10pt,a4paper,reqno]{amsart}

\usepackage[utf8]{inputenc}
\usepackage{amsmath}
\usepackage{lmodern}
\usepackage{amssymb}
\usepackage{xypic}
\usepackage{tikz-cd}
\usepackage{enumerate}

\usepackage{amsmath,amssymb,amsfonts,xfrac,MnSymbol,graphicx,float}
\usepackage{amsthm}
\usepackage{cite}
\usepackage{hyperref}
\usepackage{todonotes}
\usepackage{enumitem}
\usepackage{graphicx}
\usepackage{color}
\usepackage{import}

\numberwithin{equation}{section}

\newtheorem{theorem}{Theorem}[section]

\newtheorem{lemma}[theorem]{Lemma}
\newtheorem{corollary}[theorem]{Corollary}

\newtheorem*{claim*}{Claim}

\theoremstyle{definition}
\newtheorem{definition}[theorem]{Definition}

\newtheorem{question}[theorem]{Question}

\newtheorem*{notation*}{Notation}

\newcommand{\calR}{\mathcal{R}}

\newcommand{\calX}{\mathcal{X}}

\begin{document}

\title{Probabilistic Burnside groups} %and free subgroups \\ -- A preliminary draft --}

\author{Gil Goffer}
\address{Department of Mathematics, University of California, San Diego, La Jolla, CA 92093, USA}
\email{ggoffer@ucsd.edu}

\author{Be'eri Greenfeld}
\address{Department of Mathematics, University of Washington, Seattle, WA 98195-4350, USA}
\email{beeri.greenfeld@gmail.com}

\thanks{The authors have no conflicts of interest to declare.}

\maketitle

\begin{abstract} 
We prove that there exists a finitely generated group that satisfies a group law with probability $1$ but does not satisfy any group law. 

More precisely, we construct a finitely generated group $G$ in which the probability that a random element chosen uniformly from a finite ball in its Cayley graph, or via any non-degenerate random walk, satisfies the group law $x^k=1$ for some (fixed) integer $k$, tends to $1$. Yet, $G$ contains a non-abelian free subgroup, and therefore $G$ does not satisfy any group law. In particular, this answers two questions of Amir, Blachar, Gerasimova, and Kozma.

% for any odd, large enough $k$, we use small cancellation theory to construct a finitely generated group $G$ in which the probability that a random element $x$, chosen uniformly from a finite ball in its Cayley graph or via a non-degenerate random walk, satisfies the group law $x^k=1$ is tending to $1$. Yet, the group $G$ contains a free non-abelian subgroup. In particular, this answers a question posed by Amir, Blachar, Gerasimova, and Kozma.

%We construct probabilistic Burnside groups containing non-abelian free subgroups. These are finitely generated groups in which the probability that a random element, chosen uniformly from a finite ball in its Cayley graph or via a non-degenerate random walk, satisfies a group law of the form $x^k=1$ with probability is tending to $1$. In particular, this answers an open question posed by Amir, Blachar, Gerasimova, and Kozma, and gives the first example of a finitely generated group which satisfies a group law with probability $1$ -- but does not satisfy any group law.
\end{abstract}

\section{Introduction}

% Concise. To mention: Burnside problem(s); Burnside groups -- Ol'shanskii and previous constructions; probabilistic identities -- Eberhard, Shalev, Amir et. al.
A \textbf{law} is a non-trivial element $w\in F(x_1,x_2,\dots)$ in the free group. Given a group $G$, we say that \textbf{$G$ satisfies the law $w$}, or that \textbf{$w$ is a group law for $G$} if $w$ vanishes under any substitution of elements from $G$. For instance, the commutator $[x_1,x_2]$ is a group law for any abelian group, the iterated commutator $[x_1,[x_2,[\cdots[x_l,x_{l+1}]]$ is a group law for any $l$-step nilpotent group, and any finite group of order $k$ satisfies the law $x_1^k$.

Suppose that $G$ is equipped with a natural choice of a probability measure, $\mu$. 
For instance, the uniform measure on a finite group or more generally the Haar measure on a compact group.
% This is the case for instance when $G$ is a finite group and $\mu$ is the uniform distribution, or when $G$ is a compact group and $\mu$ is the Haar measure. 
Given a law $w$ one can now compute the probability that $w$ is satisfied in $G$, namely:
$$\mathbb{P}_{\mu}(w=1) = \mathbb{P}_{\mu}(w=1 \text{ holds in } G)=\Pr_{g_1,\dots,g_d\sim\mu}(w(g_1,\dots,g_d)=1).$$
% We sometimes write $\mathbb{P}_{\mu}(w=1)$, in short.
In various cases, laws that are satisfied in $G$ with high probability, must actually hold in full. That is: they must be actual group laws for $G$. For example, a classical theorem of Gustafson \cite{Gus} states that if $\mathbb{P}([x_1,x_2]=1)>\frac{5}{8}$ in a finite group $G$, then $G$ must be abelian.
Results in this spirit where also achieved for laws that are iterated commutators, power laws $x^k$ (see \cite{Laf1,Laf2,Laf3,Mann_Mart}), the metabelian law $[[x, y], [z, w]]$ (see \cite{DJMN}) and more.
% The picture may be slightly more complicated: 
Sometimes, the fact that a certain law $w$ holds in a finite group $G$ with probability at least $\epsilon$, implies that another law $w'$ is a group law for $G$, see for example \cite{Eber_Shum,Mann}. It is in general not known whether $\mathbb{P}(w=1 \text{ holds in } G)>\epsilon$ for a finite group $G$ always implies that $G$ admits a subgroup of index at most $d$ that satisfies a law $w'$, where $d$ and $w'$ depend only on $w$. %Is it classical or should we attribute it to someone?
In the current paper we show that this is not the case for infinite groups (see Theorem \ref{main_theorem}).

Let now $G$ be an infinite, non-compact group. There is no longer a natural probability measure on $G$ to consider. A common approach, as initiated in \cite{AMV} and adopted by \cite{ABGK}, is to take a sequence of probability measures $\vec{\mu}=\{\mu_n\}_{n=1}^{\infty}$ on $G$ whose supports exhaust $G$. We then define the probability that $G$ satisfies the law $w$ with respect to the sequence $\vec{\mu}$ by:
$$ \mathbb{P}_{\vec\mu}(\text{$w=1$ holds in $G$}):= \limsup_{n\rightarrow \infty} \Pr_{g_1,\dots,g_d\sim \mu_n}\left(w(g_1,\dots,g_d)=1\right)
$$
where in the $n$-th term of the limit, $g_1,\dots,g_d$ are independent random variables with distribution $\mu_n$.
In particular, we say that \textbf{$G$ satisfies a law $w$ with probability $1$ with respect to $\vec\mu$} if:
$$ \mathbb{P}_{\vec\mu}(\text{$w=1$ holds in $G$})=1. $$ 
Natural sequences $\vec\mu$ to consider on finitely generated infinite discrete groups are:

\begin{itemize}
    \item \textbf{Uniform measures on balls of the Cayley graph}. Fix a finite symmetric generating set $S$ of $G$ and let $B_{G,S}(n)$ denote the $n$-ball of the Cayley graph of $G$ with respect to $S$. Set $U_{B_{G,S}(n)}$ to be the uniform distribution on $B_{G,S}(n)$. We denote $\vec{\mu}_{(G,S)}=\{U_{B_{G,S}(n)}\}_{n=1}^{\infty}$.

    \item \textbf{Random walks}. Fix a finitely supported step distribution $\nu$ on $G$ whose support generates $G$ as a semigroup (`non-degenerate'). Then $\nu^{*n}$ is the $n^{th}$ step distribution with respect to a $\nu$-random walk. We denote $\vec{\mu}_{(G,\nu)}=\{\nu^{*n}\}_{n=1}^{\infty}$.
    
\end{itemize}

Another notion of probabilistic laws, applicable to residually finite groups, can be found in \cite{Shalev}. See also \cite{ABGK} for a thorough discussion on the various notions of probabilistic group laws.

As in the case of finite group, a natural question that arises is:
\begin{question}\label{que: how close is prob law to an actual law}
    Given a group that satisfies a law with probability $1$, how close is it to satisfying an actual group law?
\end{question}

Of particular interest are power laws. 
The Burnside Problem \cite{Burnside1902} asks whether a finitely generated group of finite exponent must be finite. Equivalently, whether the \textbf{free Burnside group}:
$$ B(m,k) = \left<x_1,\dots,x_m\ |\ X^k=1\ \text{for every word}\ X\ \text{in the alphabet}\ x_1,\dots,x_m\right> $$
is finite for every positive integers $m,k$. 
This problem has been settled in 1968 by Novikov and Adian \cite{NA} who proved that $B(m,k)$ can be infinite for suitable $m,k$;
see also \cite{Ad}.
Ol'shanskii \cite{Ol82,OL-book} introduced a geometric approach to this problem; see also \cite{Iv94}.

In the view of Question \ref{que: how close is prob law to an actual law}, we note the following. When $G$ is finite, for any power law $w=x^k$ there exists a threshold $0<\lambda<1$ depending on the exponent $k$ and on the number of generators of $G$, such that $\mathbb{P}_{\mu}(\text{$x^k=1$ holds in $G$})>\lambda$ implies that $x^k$ is a group law for $G$ (see \cite{Mann_Mart}).
%Let $w$ be any law such that there is some group of exponential growth satisfying $w=1$ (e.g.: the power law $w=x^k$ for $k$ large enough or the metabelian law $[[x,y],[z,w]]$).
In contrast, the following has been observed in \cite[Theorem~12.1]{ABGK}\footnote{Theorem 12.1 in \cite{ABGK} holds more generally when the power law $x^k$ is replaced by any law $w$ for which there exists a group with exponential growth satisfying $w$. We quoted only the case $w=x^k$ for simplicity.}: if we take an (infinite) Burnside group $B(m,k)$ of exponential growth (that is, sufficiently large $m,k$) and some finite group $H$ not of exponent $k$ then the group $G=B(m,k)\times H$ admits increasing generating subsets $S_\varepsilon$ such that $ \mathbb{P}_{\vec\mu_{(G,S_\varepsilon)}}(x^k=1)>1-\varepsilon $, yet $G$ does not satisfy the law $x^k=1$; it is worth noting that $G$ still satisfies some (power) law, and that the generating sets depend on $\varepsilon$ (and become larger and larger as $\varepsilon\rightarrow 0$).

% for large enough $k$, the authors of \cite{ABGK} construct an infinite group $G$ in which for any $\varepsilon>0$ there exists a finite generating set $S_{\epsilon}$ such that $G$ satisfies $ \mathbb{P}_{\vec\mu_{(G,S_\varepsilon)}}(\text{$x^k=1$ holds in $G$})>1-\varepsilon$, yet $G$ contains a free subgroup. (The results of \cite{ABGK} are more general, and hold for any law $w$ for which there exists a group with exponential growth satisfying $w$. We quoted only the case $w=x^k$ for simplicity.) However, the generating set (and hence the sequence of measures) varies as $\varepsilon$ is taken to be smaller and smaller.

Another result proven in \cite{ABGK} in this context is that group laws can hold with high probability, with respect to random walks, in a group with a free subgroup. In particular, the authors show that within the wreath product $F_2\wr \mathbb{Z}^5$ one can construct non-degenerate random walks with finitely supported step distributions such that $\mathbb{P}_{\vec\mu_{(G,\nu)}}([[x_1,x_2],[x_3,x_4]]=1)>1-\varepsilon$; this time $G$ evidently contains a free subgroup, so does not satisfy any group law. However, also here, the random walks depend (via the step distributions) on $\varepsilon$, so there is no single %step distribution that yields a 
random walk with respect to which the group satisfies a law in probability $1$.

In this paper we show that it is possible for a finitely generated group to satisfy the power law $x^k$ with probability $1$ with respect to \textbf{any} non-degenerate random walk with a finitely supported step distribution and with respect to uniform measures on balls in the Cayley graph, and yet contain a free subgroup.
In particular, this gives the first example of a group that satisfies a group law with probability $1$ but does not satisfy any group law, and answers two questions of Amir,
Blachar, Gerasimova, and Kozma (see \cite[Question~13.1]{ABGK} and the discussion thereafter, and \cite[Question~13.3]{ABGK}). 
We show:

\begin{theorem} \label{main_theorem}
There exists $k\in \mathbb{N}$ and a finitely generated group $G=\left<S\right>$, such that:
\begin{enumerate}
    \item $G$ satisfies the law $x^k=1$ with probability $1$ with respect to $\vec{\mu}_{(G,S)}$;
    \item $G$ satisfies the law $x^k=1$ with probability $1$ with respect to $\vec{\mu}_{(G,\nu)}$ for any finitely supported non-degenerated step distribution $\nu$; and yet
    \item $G$ admits a free non-abelian subgroup, and hence satisfies no group law.
\end{enumerate}
% where $S$ is a fixed finite generating set for $G$, and $\nu$ is any step-distribution on $G$ whose support is finite and generates $G$ as a semigroup.
\end{theorem}

% Note that Item (2) implies, in particular, that $G$ does not satisfy any group law.

%To a group $G$ as in Theorem \ref{main_theorem} we call a probabilistic $(m,k)$-Burnside group.
%We actually prove a slightly stronger statement, namely that for every odd, large enough $k$, there exists $M$ such that for every $m>M$ there exists a probabilistic $(m,k)$-Burnside group.

Our proof is inspired by Ol'shanskii's solution to the Burnside problem; he introduced a geometric approach to Burnside groups, constructing infinite Burnside groups
using small cancellation theory and the theory of hyperbolic groups \cite{Ol82,OL-book}.

% The Burnside Problem \cite{Burnside1902} asks whether a finitely generated group of finite exponent must be finite. Equivalently, whether the \textbf{free Burnside group}: $$ B(m,k) = \left<x_1,\dots,x_m\ |\ X^k=1\ \text{for every word}\ X\ \text{in the alphabet}\ x_1,\dots,x_m\right> $$ is finite for every positive integers $m,k$.  This problem has been settled in 1968 by Novikov and Adian \cite{NA} who proved that $B(m,k)$ can be infinite for suitable $m,k$; see also \cite{Ad}.

To prove Theorem \ref{main_theorem}, we construct a \textbf{probabilistic Burnside group} that contains a free subgroup. Namely we show:

\begin{theorem} \label{thm:main_construction}
For every odd $k>10^{10}$ and large enough $m$ there exists a group $G$ generated by the finite set $S=\{ a,b,
s_1,\dots,s_m\}$, such that:
\begin{enumerate}
    \item 
$ \mathbb{P}_{\vec{\mu}_{(G,S)}}(\text{$x^k=1$ holds in $G$}) = 1 $
\item $ \mathbb{P}_{\vec{\mu}_{(G,\nu)}}(\text{$x^k=1$ holds in $G$}) = 1 $
    \item the subgroup $\langle a,b\rangle \leq G$ is free non-abelian.
\end{enumerate}
where $\nu$ is any step-distribution on $G$ whose support is finite and generates $G$ as a semigroup.
\end{theorem}

Note that Theorem \ref{thm:main_construction} directly implies Theorem \ref{main_theorem}.

The first step in proving Theorem \ref{thm:main_construction} is to construct a group $G=\langle a,b,s_1,\dots,s_m\rangle$ that projects onto the free Burnside group $B(m+2,k)$ in which:
\begin{enumerate}
    \item the subgroup $H=\langle a,b \rangle \leq G$ is free non-abelian;
    \item every element in $G$ that is not conjugate to an element of $H$ satisfies the law $x^k=1$.
\end{enumerate}
In fact, $G$ is a `minimal cover' of the free Burnside group $B(m+2,k)$ that contains a free subgroup.
Interesting constructions of `almost periodic quotients' have been developed in \cite{Coulon_Burnside_Quotients}.

%However... I'm not sure that however is in place here
A major challenge is bounding the relative growth of the conjugates of $H$ in $G$. Namely, we need to show that `almost all' elements in sufficiently large balls in the Cayley graph of $G$ are not conjugate to $H$ and thus satisfy $x^k=1$. This guarantees that $G$ satisfies the law $x^k=1$ with probability $1$.
%, yet maintaining the above Burnside type property.
In general, even the conjugates of small subgroups can occupy a large portion of the ambient group; a surprisingly extreme situation of a group in which this phenomenon occurs is given in \cite{Osin}.
In order to make sure that the union of conjugates of $H$ in $G$ remains tame, %H$ must be carefully constructed. In more detail,
we show that $H$ performs a certain convexity-like behavior: the shortest presentation in $G$ of an element of $H$ is always a word in the generators of $H$. Moreover, the shortest presentation in $G$ of a conjugate of an element of $H$ always takes the form $ghg^{-1}$ where $h$ is a word in the generators of $H$, and $g\in G$ is arbitrary. This requires delicate geometric analysis of a limit small cancellation construction.

The methods we use are derived from Ol'shanskii's work but require further layers of construction in order to make sure that the probabilistic Burnside-type property and a free subgroup coexist.

The structure of the article is as follows: Section \ref{sec:preliminaries} consist of the relevant preliminaries from \cite{OL-book}; Section \ref{sec:construction and proofs} introduces the main construction; Section \ref{sec:condition A} proves the necessary geometric properties of maps over the presentation we construct; Section \ref{sec:proof of main theorem} verifies that the law $x^k=1$ holds in the constructed group with probability $1$, thus completing the proof of Theorem \ref{thm:main_construction}. Lastly, the appendix proves the main small cancellation feature of the construction, which is essential to proving Theorem \ref{thm:main_construction}, and includes a road-map for points of resemblance and difference between our construction and the construction in \cite{OL-book}.

\subsection*{Acknowledgements.} 
We are gratefully thankful to Denis Osin for inspiring discussions, and in particular for his advice regarding Lemma \ref{lem:conjugates of K have length 2V+K}; and to Rémi Coulon for bringing the interesting construction in \cite{Coulon_Burnside_Quotients} to our attention.

% We would like to express our sincere thanks to Denis Osin for multiple advice on the implementation of the small cancellation theory in our project, and especially regarding Lemma \ref{lem:conjugates of K have length 2V+K} and its proof.

\section{Diagrams over groups}
\label{sec:preliminaries}

The definitions and results in this section are taken from \cite{OL-book}.

\subsection{Letters and words}
Let $S=\{s_1,\dots,s_d\}$ be a fixed alphabet and let $S^{-1}=\{s_1^{-1},\dots,s_d^{-1}\}$ be a new alphabet. A \emph{word} in the alphabet $S\cup S^{-1}$ is a finite sequence $X=x_1x_2\dots x_l$ with each $x_i\in S\cup S^{-1}$. The number $l$ is called \emph{the length} of $X$ and is denoted by $|X|$. The empty word has length $0$ by convention. $X$ is called \emph{reduced} if it does not contain any occurrence of $xx^{-1}$ or $x^{-1}x$ for any $x\in S$. We denote by $F=F(S)$ the set of all reduced words over the alphabet $S \cup S^{-1}$. We use the same notation, $F(S)$, for the free group over the alphabet $S$.
Two (not necessarily reduced) words $X,Y$ that are the same letter by letter are said to be \emph{visually equal}, and we then write $X\equiv Y$. A reduced word $Y$ is called a \emph{subword} of another reduced word $X$ if $X\equiv UYV$ for some words $U,V$. 

%When $U$ and $V$ are equal in the free group $F(S)$ we say that they are \emph{visually equal}, and we write $U \equiv V$.
%For $W\in F(S)$ we say that $W=1$ is a \emph{consequence of $\calR$} if $W=1$ in $G=\langle S \mid \calR \rangle$. This is equivalent to saying that $W$ belongs to the normal closure of $\calR$ in $F(S)$.

If $X\equiv YZ$ is a reduced word then $ZY$ is called a \emph{cyclic shift} of $X$. Clearly cyclic shifts of $X$ are conjugate to $X$ in $F=F(S)$ and every word is conjugate in $F$ to a \emph{cyclically reduced word}, that is, a reduced word whose first and last letters are not mutual inverses. Moreover, if $X,X'$ are two cyclically reduced words that are conjugate in $F$ then $X'$ is a cyclic shift of $X$.
A \emph{cyclic word} is the set of all cyclic shifts of a word $X$. A \emph{subword} of a cyclic word is a subword of one of the cyclic shifts of the ordinary word $X$.

%We say that $X$ is a \emph{proper power} in $F=F(S)$ if $X= Y^l$ for some $Y$ with $l>1$. We call a non-empty cyclically reduced word $X$ \emph{visually-simple} if it is not a proper power in $F$. It is obvious that every cyclic shift of a visually-simple word is itself visually-simple and every non-trivial word in $F$ is conjugate in $F$ to a nonzero power of a visually-simple word.

%Any subword of a power of $A^m$, $m>0$, is called a \emph{periodic word with period $A$} (an $A$-periodic word). E.g., $ababa$ is a periodic word with periods $ab,ba$ but also with period $ababab^{-1}$. We consider periodic words with visually-simple periods only (e.g., $abab$ is not considered a period).

%Let $X,Y$ be two $A$-periodic words with $|X|,|Y|\geq A$. We say that decompositions $X\equiv X_1X_2,Y\equiv Y_1Y_2$ are $A$-compatible if, for some power $A^m,m>0$ we have $A\equiv U_1U_2$ such that $U_1\equiv V_1X_1\equiv W_1Y_1$ and $U_2\equiv X_2V_2\equiv Y_2W_2$ for some words $V_1,V_2,W_1,W_2$. It is easy to verify that if $X,Y$ are $A$-compatible then they are $A'$-compatible for every cyclic shift $A'$ of $A$. For any $l\geq 2$ we call a word $X$ \emph{$l$-aperiodic} if it has no non-empty subwords of the form $Y^l$.

%Theorem 4.6 in Ol'shanskii's book states that there exists arbitrarily long $6$-aperiodic words in the alphabet $\{a,b\}$. In fact, the number of such words of length $n$ is greater than $(3/2)^n$.

\subsection{Maps}
A cell-decomposition $\Delta$ of a surface $X$ is called a \emph{map} on $X$. 
When $X$ is an orientable surface, we assume that an orientation has been fixed on $\Delta$.
Maps on a disc are called \emph{circular} and maps on an annulus are called \emph{annular}. 
Any side of $\Delta$ with a specified orientation is called an \emph{edge} in $\Delta$, that is, every side gives rise to two edges in $\Delta$, which are called inverses of each other. Let $e_-,e_+$ denote respectively the uniquely determined vertices of the edge $e$, so that $(e^{-1})_-=e_+$ and $(e^{-1})_+=e_-$. 
%The oriented edges decomposing the surface into cells are called the \emph{edges} of the map. For every edge $e$ associated is an edge $e^{-1}$ with the opposite orientation.
If $e_1,\dots,e_n$ are edges in $\Delta$ such that $(e_i)_+=(e_{i+1})_-$ for $i=1,\dots,n-1$, then the sequence $p=e_1\dots e_n$ is called a \emph{path} in $\Delta$. We denote the length of such a path by $|p|=n$.
The concept of \emph{subpath} is fully analogous to that of a subword: $p$ is a subpath of $q$ if $q=p_1pp_2$ for some paths $p_1,p_2$.

%A \emph{loop} is a path $e_1e_2\dots e_n$ in which $(e_1)_-=(e_n)_+$.
A path $p$ is called \emph{simple} if the vertices $(e_1)_-,\dots,(e_n)_-,(e_n)_+$ are pairwise distinct. A path $p$ is called \emph{geodesic} if $|p|\leq |p'|$ for any path $p'$ that can be achieved from $p$ by performing finitely many insertions of loops, deletions of loops, insertions of subpaths of the form $ee^{-1}$, and deletions of subpaths of the form $ee^{-1}$.

By convention, the boundary of a circular map is traversed clockwise, and for an annular map, the outer is traversed clockwise while the inner is traversed counter clockwise. Any connected component of the boundary of a map %consisting of $n$ unoriented edges $e_1,e_2,\dots,e_n$
that form a loop is called a \emph{contour} of the map; the contour(s) of a map $\Delta$ are denoted $\partial \Delta$; circular maps have one contour and annular maps have two. The same terminology and notation is used for the contour $\partial \Pi$ of a cell $\Pi$. % A contour is regarded to as defined up to a cyclic shift.
We consider contours up to cyclic shifts, namely, we do not specify a unique starting point.

%If an edge $e_i$ occurs in a contour $p=e_1\dots e_n$ of a cell $\Pi$ (from a map $\Delta$) then we say that $e_i$ \emph{belongs} to $p$.
%Here $e_i^{-1}$ may not belong to $\partial \Pi$ (by which we denote the contour of $\Pi$) or to $\partial \Delta$. In the case when $e$ o $e^{-1}$ belongs to $\partial \Pi$ (or $\partial \Delta$) we say that $e$ \emph{lies} on $p$.

\subsection{Diagrams over groups}
%For an alphabet $S$ we denote $S^1=S\cup S ^{-1}\cup \{1\}$. A map $\Delta$ equipped with an assignment of a letter $\phi(e)\in S^1$ for each edge $e$ of $\Delta$ such that $\phi(e^{-1})\equiv \phi(e)^{-1}$ is called a \emph{diagram over $S$}. %(Of course, $1^{-1}\equiv 1$.)
%The label of a path $p=e_1\dots e_n$ is set to be the word $\phi(p):=\phi(e_1)\dots\phi(e_n)$ over $S^1$; if $|p|=0$ then $\phi(p)\equiv 1$ by definition. The label of a contour (of a cell or diagram) is defined up to a cyclic shift, that is, as a cyclic word.
A map $\Delta$ equipped with an assignment of a letter $\phi(e)\in S\cup S^{-1}$ for each edge $e$ of $\Delta$ such that $\phi(e^{-1})\equiv \phi(e)^{-1}$ is called a \emph{diagram over $S$}. %(Of course, $1^{-1}\equiv 1$.)
The label of a path $p=e_1\dots e_n$ is set to be the word $\phi(p):=\phi(e_1)\dots\phi(e_n)$ over $S$; if $|p|=0$ then $\phi(p)\equiv 1$ by definition. The label of a contour (of a cell or diagram) is defined up to a cyclic shift, that is, as a cyclic word.

Let $\calR\subset F(S)$ be a set of words. %The \emph{normal closure} $\langle\langle \calR \rangle \rangle$ of $\calR$ is the smallest normal subgroup of $F$ containing $\calR$ and the group $\langle S \mid \calR \rangle$ is $F/\langle\langle \calR \rangle \rangle$. Given a group $G$, an isomorphism $G\cong \langle S \mid \calR \rangle$ is called a \emph{presentation} of $G$ with generator set $S$ and relator set $\calR$.
%We say that $W$ and $V$ are $\calR$-equivalent if $W=V$ in the group $G=\langle S \mid \calR \rangle$, and sometimes simply write `$W=V$ in $G$', as this equality depends only on $G$ rather than on $\calR$.
% Now let $G$ be any group, given by the presentation \begin{equation}\label{eq:G=<S|R>} \langle S \mid R=1\ :\ R\in \calR \rangle.  \end{equation}
A cell of a diagram $\Delta$ is called an $\calR$-cell if the label $\phi(p)$ of its contour is equal up to a cyclic shift to some $R\in \calR\cup \calR^{-1}$. %, or to $R^{-1}$. %or to some word obtained from $R$ or $R^{-1}$ by insertions of the symbol $1$.

%A \emph{$0$-cell} is a cell in a diagram whose label of its contour $e_1\dots e_n$ is visually equal to either $\phi(e_1)\dots \phi(e_n)$ with each $\phi(e_i)\equiv 1$, or for some $i\neq j$, $\phi(e_i)=a,\phi(e_j)=a^{-1}$ and $\phi(e_k)=1$ for all $k\neq i,j$.
%Edges with label $1$ are called \emph{$0$-edges} and those with labels from $S^{\pm 1}$ are called \emph{$S$-edges}. The length $|p|$ of a path $p$ is defined to be the number of its $S$-edges.
%The perimeter $|\partial \Pi|$ of a  cell $\Pi$ is the length of its contour. The perimeter $|\partial \Delta|$ is defined similarly.

A \emph{diagram %on a surface $X$ 
over a presentation $G=\langle S \mid \calR\rangle$} is any diagram $\Delta$ over the alphabet $S$ whose cells are all $\calR$-cells. %or $0$-cells.
We also say that $\Delta$ is a \emph{diagram over $G$}.
Ol'shanskii \cite{OL-book} calls a diagram over a presentation \emph{reduced} if it does not contain a pattern called `$j$-pair' %for any rank $j$,
and shows that any diagram $\Delta$ over a given presentation can be replaced by a reduced diagram with the same contour label. %However, to avoid complications and the need of defining $j$-pairs,
We use the following equivalent definition by Osin \cite{Osin}:
A diagram $\Delta$ over a presentation $G=\langle S \mid \calR \rangle$ is called \emph{reduced} if it has a minimal number of $\calR$-cells among all diagrams over $G$ having the same boundary label as $\Delta$.
\subsection{Graded presentations and graded diagrams}
Suppose that a set of words $\calR$ in $F(S)$ is decomposed into %a disjoint union of
subsets $\calR=\bigcup_{i=1}^{\infty}\calR_i$ in such a way that $\calR_i$ is disjoint from the set of cyclic conjugates and inverses of $\calR_j$ for any $j\neq i$.
% with a cyclic conjugate of a word in $\calR_j$ or its inverse when $j\neq i$. 
Then the presentation  $G=\langle S \mid \bigcup_{i=1}^{\infty}\calR_i\rangle$
% $G=\langle S \mid R=1, R\in \calR=\bigcup_{i=1}^{\infty}\calR_i\rangle$ 
is called a \emph{graded presentation}, and the members of $\calR_i$ are called \emph{relators of rank $i$}. If two words $X,Y$ are equal in the group $G(i):=\langle S \mid \bigcup_{j=1}^{i}\calR_j\rangle$,
% $G(i)=\langle S \mid R=1, R\in \calR_i=\bigcup_{j=1}^{i}\calR_i\rangle$,
we say that they are \emph{equal in rank $i$}. By convention, $G(0)=F(S)$, so `equality of rank $0$' of reduced words $X,Y$ amounts to $X\equiv Y$. Conjugacy in rank $i$ is defined in a similar manner.

A \emph{graded map} $\Delta$ is a map on a surface together with an assignment of an integer $r(\Pi)\geq 0$ to any cell $\Pi$ of $\Delta$. We call $r(\Pi)$ the \emph{rank} of $\Pi$. %if the set of its cells decomposes into sets of cells of equal rank. 
%If $\Pi$ is a cell of rank $i$, we write $r(\Pi)=i$.
A diagram $\Delta$ is called \emph{graded} if its underlying map is graded (i.e., the map obtained by forgetting its labeling). Naturally, for a diagram $\Delta$ over a graded presentation, we set $r(\Pi)=i$ if the contour label of the cell $\Pi$ is visually equal to a relator of rank $i$. %Moreover, the cells of rank $0$ are just the $0$-cells (in the sense of chapter 11.2 in the book)\gil{need to decide if we want to explain 0 cells or not}.

The rank of a map (or a diagram) is: $r(\Delta)=\max_{\Pi\in \Delta} r(\Pi)$. %If $r(\Delta)\leq i$ we say that $\Delta$ is a map (diagram) \emph{of rank $i$}. % If the maximum rank of a cell in $\Delta$ is equal to $i$, we write $r(\Delta)=i$.
The following two Van-Kampen type theorems are proven by Ol'shanskii in \cite{OL-book}.

%Consider the presentation $(*)=G=\langle S \mid R=1, R\in \calR=\bigcup \calR_i\rangle$

\begin{lemma}[{\cite[Theorem 13.1]{OL-book}}]\label{lem:thm 13.1 van kampen}
    Let $W$ be a non-empty word in $F(S)$. Then $W=1$ in the group $G$ given by a graded presentation if and only if there is a reduced circular diagram over this presentation, whose contour label is visually equal to $W$.
\end{lemma}

\begin{lemma}[{\cite[Theorem 13.2]{OL-book}}] \label{lem:thm 13.2 annular van-kampen}
    Let $V$ and $W$ be two non-empty words in $F(S)$. Then these words are conjugate in the group $G$ given by a graded presentation if and only if there is a reduced annular diagram over this presentation, whose contours $p$ and $q$ satisfy $\phi(p)\equiv V$ and $\phi(q)\equiv W^{-1}$.
\end{lemma}

%In what follows, all maps and diagrams are assumed to be graded, so the adjective `graded' is often omitted.

\subsection{Contiguity subdiagrams}
Let $p$ be a loop on a surface $X$ whose edges form the boundary of a subspace $Y\subset X$ homeomorphic to a disc. Then the restriction of a cell decomposition $\Delta$ of $X$ to $Y$ is a decomposition of $Y$, called a \emph{submap} of the map $\Delta$. %We often specify submaps by their contours.
By definition, a submap is always a circular submap.
A subdiagram of a diagram $\Delta$ is a submap $\Gamma$ of $\Delta$ whose edges bear the same labels as in $\Gamma$. It can be thought of as a circular diagram cut out from $\Gamma$.

Let $\Delta$ be a diagram over a presentation $\langle S \mid \calR \rangle$ and let $q$ be a subpath of $\partial \Delta$. Let $\pi,\Pi$ be $\calR$-cells in $\Delta$. Suppose that there is a simple closed path $p=s_1q_1s_2q_2$ in $\Delta$, where $q_1$ is a subpath of $\partial \pi$ and $q_2$ a subpath of $q$ (or of $\partial \Pi$), and $|s_1|,|s_2|$ are bounded by a small fixed constant. We denote by $\Gamma$ the subdiagram of $\Delta$ inscribed by $p$. If $\Gamma$ contains no cells of rank $\geq i$ then $\Gamma$ is called a \emph{contiguity subdiagram of rank $i$} of $\pi$ to the subpath $q$ of $\partial \Delta$ (or to $\Pi$, respectively). $q_1$ is called the \emph{contiguity arc} of $\pi$ to $q$ (respectively to $\Pi$), and the ratio $|q_1|/|\partial \pi|$ is called the \emph{contiguity degree} of $\pi$ to $q$ (resp. to $\Pi$) and is denoted by $(\pi,\Gamma,q)$ (resp. $(\pi,\Gamma,\Pi)$). We sometimes write $q_1=\Gamma\cap \pi$ and $q_2=\Gamma\cap q$ (resp. $q_2=\Gamma\cap \Pi)$).
For a detailed inductive construction of contiguity subdiagrams, we refer the reader to \cite[Section~14.2]{OL-book}. %(And given in the appendix below?).

\section{The construction}\label{sec:construction and proofs}
%We use Ol'shanskii's terminology from \cite[Chapter 6]{OL-book}, adjusted to our needs.
Fix an alphabet $S=\{a,b,s_1,s_2,\dots,s_m\}$. %and $F=\langle S \rangle$ be the free group over $S$.
Let $G(0)=F(S)$ be the free group over $S$, and set $\calX_0=\calR_0=\emptyset$. 
We now recursively define for any integer $i\geq 0$ finite sets $\calX_i,\calR_i$ of words over $S$ and a finitely presented group $$G(i)=\langle S \mid R=1 \ :\ R\in\bigcup _{j=0}^{j=i}\calR_j \rangle. $$

% which is the quotient of $F(S)$ by the normal closure of $\bigcup _{j=0}^{j=i}\calR_i$. 

%Let $X$ be a word over $S$. If $X^n=1$ is a relation of rank $i$, namely, the equality $X^k=1$ holds in $G(i)$ and not in $G(i-1)$, $i\geq 1$, then $X$ is called a
%A word $X\in \calX_l$ is called a \textbf{period of rank $l$}.

Let $i\geq 0$ and suppose that $G(i)$, and $\calX_l$ for any $l\leq i$, have already been defined. A word of $\calX_l$ is called a \textbf{period of rank $l$}.
\begin{definition}\label{def:simple}
A (nontrivial) word $W$ over $S$ is called \textbf{simple in rank $i$} if:
\begin{enumerate}[label=(S\arabic*)]
    \item $W$ is not conjugate in %rank $i$
    $G(i)$ to any power $B^m$ of a period $B$ of rank $l\leq i$;
    \item $W$ is not conjugate in %rank $i$
    $G(i)$ to a power of any word $C$ with $|C|<|W|$; and
    \item $W$ is not conjugate in %rank $i$
    $G(i)$ to any word in the subgroup $\langle a,b\rangle$ of $G(i)$.\footnote{Note that $G(i)$ is a quotient of $F(S)=F(\{a,b,s_1,s_2,\dots,s_m\})$. By $\langle a,b\rangle$ we mean the subgroup of $G(i)$ generated by $a,b$.}
\end{enumerate}
\end{definition}
Let $\calX_{i+1}$ denote a maximal subset of words simple in rank $i$, of length $i+1$, and with the property that if $A,B\in \mathcal{X}_{i+1}$ with $A\neq B$ then $A$ is not conjugate in rank $i$ (that is, in $G(i)$) to $B$ or $B^{-1}$. 
Set $\calR_{i+1}$ to be $\{X^k \mid X\in \calX_{i+1}\}$, and set 
$$ G(i+1)=\langle S \mid R=1 \ :\ R\in\bigcup _{j=0}^{j=i+1}\calR_j \rangle. $$
Finally, set $\calR=\bigcup_{j=0}^{\infty} \calR_j$ and consider the group $G$ given by the graded presentation

\begin{equation}\label{eq:G=<S|bigcup R_i>}
    G=\langle S \mid R=1\ :\ R\in \calR=\bigcup_{j=0}^{\infty}\calR_j\rangle
\end{equation}

We comment that in Ol'shanskii's original construction of a Burnside group \cite[Chapter~6]{OL-book}, the definition of simple elements is slightly different. In fact, simple elements of \cite{OL-book} are elements that satisfy Items (S1) and (S2) above. This, in turn, yields a slightly different definition for a period.

%The rest of this section is dedicated to proving the properties of $G$.

\begin{lemma} \label{lem:torsion}
Let $g\in G$. Then either $g$ is conjugate in $G$ to an element from the subgroup $\langle a,b\rangle$ or $g^k=1$.
\end{lemma}

\begin{proof}
This follows directly from the construction, but we prove it here by induction on the norm for the convenience of the reader. Indeed, suppose that an element $g\in G$ is not conjugate to any element in $\langle a,b\rangle$. Let $i=\|g\|$ be the norm of $g$, % and suppose that any element with norm less than $i$ is either of order dividing $k$ or is conjugate to some element from $\langle a,b\rangle$.  
and let $W$ be a word of length $i$ representing $g$ in $G$. By assumption, $W$ satisfies Item (S3) in the definition of a simple element of rank $i-1$. If $W^k\neq 1$ in $G$, then $W$ also satisfies (S1) and (S2), since by induction hypothesis, any word $C$ with $|C|<|W|$ is either of order dividing $k$ or is conjugate to some element from $\langle a,b \rangle$. Therefore $W$ is simple of rank $i-1$, and so $\calX_{i}$ contains some $X$ that is conjugate to $W^{\pm1}$. But then $X^k\in \calR_i$ implies that $X^k=1$ in $G$, and therefore also
$W^k=1$, deriving a contradiction.
\end{proof}

Since all of the defining relations of $G$ are of the form $w^k$, we obtain:

\begin{corollary} \label{cor:quotient burnside}
There is a natural isomorphism $G/\left<\left< a,b \right> \right>\cong B(m,k)$.\footnote{$\langle \langle a,b\rangle \rangle$ denotes the normal closure of the subgroup $\langle a,b\rangle$.}
\end{corollary}

\section{Condition $A$ for maps over Presentation \ref{eq:G=<S|bigcup R_i>}}\label{sec:condition A}

\subsection{Condition $A$ and $\gamma$-cells}

%\subsection{$A$-maps and existence of $\gamma$-cells}
We now describe a certain condition on maps, defined by Ol'shanskii \cite[Section §15.2]{OL-book}, called \emph{Condition $A$}.
Intuitively, this condition should capture the special features of a map over a presentation whose relators are of the form $X^k$, $k$ large. %in which all rank $j$ relators are of the form $B^k$ for $B$ a simple period of length $j$, perThe precise definition of a simple period and precise choice of relators in our construction, are brought in Section \ref{sec:construction and proofs}.

The definition of an $A$-map (namely, a map that satisfies Condition $A$) and the facts on $A$-maps here to follow hold for a sufficiently large\footnote{That is, $k>10^{10}$.} odd integer $k$ along with small positive parameters,
$\alpha>\gamma>\epsilon>0$ defined in \cite{OL-book}.
We denote $\bar{\alpha}=\frac{1}{2}+\alpha,$ %$\bar{\beta}=1-\beta, 
$\bar{\gamma}=1-\gamma$, and note that the parameters are chosen such that $\bar{\alpha}+\epsilon<\bar{\gamma}$. %These notations are made to maintain coherence with \cite{OL-book}, in order to make the references easier to follow.

\begin{definition}[{Condition A}]\label{def:condition A} \label{def_A}
A graded map $\Delta$ is called an \emph{$A$-map} if:
\begin{enumerate}[label=(A\arabic*)]
        \item The contour $\partial\Pi$ of any cell $\Pi$ of rank $j$ is cyclically reduced (that is, does not contain a subpath of the form $ee^{-1}$) and $|\partial\Pi|\geq kj$.
        \item Any subpath of length $\leq \max(j,2)$ of the contour of any cell of rank $j$ is geodesic in $\Delta$.
        \item If $\pi,\Pi$ are two cells in $\Delta$ and $\Gamma$ is a contiguity submap of $\pi$ to $\Pi$ with $(\pi,\Gamma,\Pi)\geq \epsilon$, then $|\Gamma\cap\Pi|<(1+\gamma)\cdot r(\Pi)$.
    \end{enumerate}
\end{definition}

Recall that in a map $\Delta$ over a graded presentation, the rank of a cell $\Pi$
is $r(\Pi)=i\geq 1$ if the label of the contour of $\Pi$ is equal in $F(S)$ to a word from $\calR_i$, or its inverse, or a cyclic conjugate of either. %In particular, in the presentation \ref{eq:G=<S|bigcup R_i>}, we have $|\partial \Pi|=ki$ if $r(\Pi)=i\geq 1$ as $|A^k|=ki$ for some $A\in \calX_i$. Every $\calR$-cell of rank $i$ corresponds uniquely to a period $A$ of rank $i$. %(We say that $\Pi$ is a cell with period $A$.)
We will now show that any map over the graded presentation \ref{eq:G=<S|bigcup R_i>} %constructed in Section \ref{sec:construction and proofs}
is an $A$-map.  Indeed, as any rank $j$ relator in our construction is of the form $B^k$ for a period $B$ of length $j$, $(A1)$ is satisfied; $(A2)$ is satisfied since the periods of rank $j$ are chosen as the shortest words among their conjugates in rank $j-1$. Property $(A3)$ generalizes a certain observation on periodic words, see \cite[Section §15.2]{OL-book}.
%the following observation: given any word $X$ which is both $B$-periodic and $C$-periodic with $B\neq C$ and $X$ is `much longer' than $B$ -- say, $|X|>k\epsilon |B|$ -- we must have that $|X|>(1+\gamma)|C|$ where $\gamma$ is `small'. Viewing the word $X$ as the label of the contiguity arc of $\Gamma$, we get property $(A3)$.

\begin{lemma}[{Analogous to \cite[Lemma 19.4]{OL-book}}]\label{lem:our 19.4 the construction has property A}
Every reduced diagram of finite rank over Presentation \ref{eq:G=<S|bigcup R_i>} is an $A$-map. 
\end{lemma}
\begin{proof}
Ol'shanskii's Lemma 19.4 in \cite{OL-book} states that every reduced diagram of finite rank over the graded presentation brought in \cite[Section §18]{OL-book} is an $A$-map. Although Lemma 19.4 \cite{OL-book} is formulated for a specific presentation, its proof uses only properties $(S1)$ and $(S2)$ of Definition \ref{def:simple}, and the following four properties of the presentation:
    \begin{enumerate}[label=(P\arabic*)]
        \item Every relator of rank  $i$ is of the form $A^k$ for a period $A$ of rank $i$ and $|A|=i$;
        \item A period of rank $i$ is simple in any rank less than $i$;
        \item A period of rank $i$ is the shortest word among its conjugates in rank $i-1$;
        \item For two periods $A,B$ of ranks $i,j$ respectively, $A$ is not conjugate to $B^{\pm 1}$ in any rank smaller than $\min(i,j)$, unless $A\equiv B$.
    \end{enumerate}
%and the fact that every simple element satisfies properties $(S1)$ and $(S2)$ of definition \ref{def:simple}.
It therefore suffices to show that Presentation \ref{eq:G=<S|bigcup R_i>} satisfies $(P1)-(P4)$. The deduction of the conclusion from these four properties works precisely as Ol'shanskii's proof of Lemma 19.4 in \cite{OL-book}.
We now verify properties $(P1)-(P4)$.

Property $(P1)$ follows directly from the definitions of $\calX_i$, $\calR_i$.

For $(P2)$, observe that a simple element of rank $i-1$ is in particular simple in any rank $j\leq i-1$. $(P2)$ now follows directly from the fact that $\calX_i$ contains only elements that are simple in rank $i-1$.

$(P3)$ follows immediately from $(P2)$. Indeed, a period of rank $i$ is simple in any rank $<i$, so it is not conjugate in $G(i-1)$ to any shorter word by Item (S2) of Definition \ref{def:simple}.

For $(P4)$ assume without loss of generality that $i\leq j$. If $i<j$, then either $B$ or $B^{\pm 1}$ is conjugate in $G(i)$ to the period $A$, and so $B$ is not simple in rank $j$, contradicting $(P2)$. If $i=j$, then $A\equiv B$ by definition of the set $\calX_i$, concluding the proof of $(P4)$.
\end{proof}
For completeness, a full proof of Lemma \ref{lem:our 19.4 the construction has property A} in brought in the Appendix, see Lemma \ref{lem:lem 19.4}

Properties $(A2)$ and $(A3)$ of subpaths of contours of cells from Definition \ref{def:condition A} extend, when discussing contiguity submaps, to sections of contours of these submaps. %in the definition of condition $A$ The concept of a rank of a cell is generalized to sections of the contour of $\Delta$ as follows.
\begin{definition}
A (cyclic) section $q$ of a contour of a map $\Delta$ is called a \emph{smooth section of rank $r>0$} (we write $r(q)=r)$ if:
\begin{enumerate}
    \item Every subpath of length $\leq \max(r,2)$ of $q$ is geodesic in $\Delta$; and
    \item For each contiguity submap $\Gamma$ of a cell $\pi$ to $q$ satisfying $(\pi,\Gamma,q)\geq \epsilon$, we have $|\Gamma\cap q|<(1+\gamma)r$.
\end{enumerate}
\end{definition}
Intuitively this means that the section $q$ looks locally like a section in the contour of a cell of rank $r$. To avoid ambiguity, for a given section $q$ we let its rank be the minimal $r$ for which this definition is satisfied.

The next three lemmas about $A$-maps are taken from \cite{OL-book}.

\begin{lemma}[{\cite[Lemma~15.8]{OL-book}}]\label{lem:contiguity degree to smooth section is always small}
    In an arbitrary $A$-map $\Delta$, the degree of contiguity of an arbitrary cell $\pi$ to an arbitrary cell $\Pi$ or to an arbitrary smooth section $q$ of the contour across an arbitrary contiguity submap $\Gamma$ is smaller than $\bar{\alpha}$.
    %If also $r(\Pi)\leq r(\pi)$ (or $r(q)\leq r(\pi)$) then the the degree of continuity is less than $\epsilon$
\end{lemma}

\begin{lemma}[{\cite[Lemma~16.1]{OL-book}}]\label{lem:Greendlinger-like lemma}
Let $\Delta$ be a circular $A$-map, $r(\Delta)>0$, whose contour is decomposed into four sections $q^1,q^2,q^3,q^4$. Then there is an $\calR$-cell $\pi$ and disjoint contiguity submaps $\Gamma_1,\Gamma_2,\Gamma_3,\Gamma_4$ of $\pi$ to $q^1,q^2,q^3,q^4$, respectively, in $\Delta$ (some of the latter may be absent) such that: 
$$\Sigma_{i=1}^{4}(\pi,\Gamma_i,q^i)>\bar{\gamma}.$$
\end{lemma}

\begin{corollary}[{\cite[Corollary~16.2]{OL-book}}]\label{lem:annular Greendlinger-like lemma}
    Let $\Delta$ be an annular $A$-map with contours $q^1$ and $q^2$ (regarded as cyclic sections) and $r(\Delta)>0$. Then $\Delta$ has an $\calR$-cell $\pi$ and disjoint contiguity submaps $\Gamma_1$ and $\Gamma_2$ of $\pi$ to $q^1$ and $q^2$, respectively, (one of these may be absent) such that $(\pi,\Gamma_1,q^1)+(\pi,\Gamma_2,q^2)>\bar{\gamma}.$
\end{corollary}

Any cell $\pi$ as in Theorems \ref{lem:Greendlinger-like lemma} and \ref{lem:annular Greendlinger-like lemma} is called a \emph{$\gamma$-cell} in $\Delta$. Intuitively, these theorems say that almost all of the contour of a $\gamma$-cell consists of outer arcs.

Finally, we will use the following claim, that is proved by Ol'shanskii as part of his proof of \cite[Theorem 16.2]{OL-book}.
\begin{lemma}\label{lem:if Gamma is with minimal num of cells then r(Gamma)=0}
    Let $\Delta$ be an $A$-map with $r(\Delta)>0$ and let $\Gamma$ be a contiguity submap of an $\calR$-cell $\Pi$ to a section $q$ of the contour of $\Delta$ such that $(\Pi,\Gamma,q)\geq \epsilon$. Suppose that $\Pi$ is chosen in such a way that the number of cells in $\Gamma$ is minimal, then $r(\Gamma)=0$.
\end{lemma}

\subsection{The subgroup $\langle a,b \rangle$}

%We recall that by Lemma \ref{lem:our 19.4 the construction has property A}, every map over the presentation \ref{eq:G=<S|bigcup R_i>} is an $A$-map. This allow us to analyze in detail certain convexity-like behaviours of the subgroup $\langle a,b\rangle$. This behaviour enable us to bound the asymptotic density of the set of non-torsion elements in $G$. The rest of this section is dedicated to the analysis of the geometry of the subgroup $\langle a,b\rangle$. 

The fact that every map over Presentation \ref{eq:G=<S|bigcup R_i>} is an $A$-map allows us to analyse in detail a certain convexity-like behaviour of the subgroup $\langle a,b\rangle$. The convexity-like behaviour enables us to bound the asymptotic density of the set of non-torsion elements in $G$. The rest of this section is dedicated to the analysis of the geometry of the subgroup $\langle a,b\rangle$.

 %Its proof is inspired by Ol'shanskii's proof of Theorem 16.2 in \cite{OL-book}, and uses the notions of $k$-bonds and principal cells. Intuitively, bonds are the building blocks over which contiguity submaps are built, and $k+1$-bonds are built from $k$-bonds using principal cells. For a detailed introduction and precise definitions of principal cells and of bonds we refer the reader to Section 14.2 in \cite{OL-book}.

The following lemma formalizes the fact that relators from $\calR$ have very small coincidence with words over the alphabet $\{a,b\}$.

\begin{lemma}\label{lem:contiguity degree to <a,b> is small}
    Let $\Delta$ be a reduced diagram over Presentation \ref{eq:G=<S|bigcup R_i>} and $\Pi$ a cell in $\Delta$. Let $q$ be a section of the contour of $\Delta$ such that the label $\phi(q)$ is a word from $F(a,b)$. Then any contiguity submap $\Gamma$ of $\Pi$ to $q$ has contiguity degree $(\Pi,\Gamma,q)< \epsilon$.
\end{lemma}

\begin{proof}
Suppose $\Gamma$ is a contiguity submap of $\Pi$ to $q$ with $(\Pi,\Gamma,q)\geq \epsilon$ and such that the number of cells in $\Gamma$ is minimal. %The bonds defining $\Gamma$ are $0$-bonds, as otherwise such a bond %(as can be seen from the inductive definition of a bond) 
%would contain a principal cell $\pi$ of a bond whose contiguity degree to $q$ over a contiguity submap $\Gamma'$ is at least $\epsilon$ (by the definition of a bond), while the number of cells is $\Gamma'$ is strictly smaller than in $\Gamma$. Hence the contour of $\Gamma$ has the form $p_1q_1p_2q_2$ where $|p_1|=|p_2|=0$, $q_1=\Gamma\cap \Pi$, $q_2=\Gamma\cap q$.
%Suppose now that $r(\Gamma)>0$, and let $\pi$ be a cell in $\Gamma$. By the minimality assumption on $\Gamma$, the contiguity degree of $\pi$ to $q_2$ is less than $\epsilon$. By Lemma \ref{lem:contiguity degree to smooth section is always small} the contiguity degree of $\pi$ to $q_1$ is at most $\bar{\alpha}$. It follows that the contiguity degrees of $\pi$ to $\partial \Gamma$ sum to at most $\epsilon + \bar{\alpha}<\bar{\gamma}$, contradicting Lemma \ref{lem:Greendlinger-like lemma}. We therefore get that $r(\Gamma)=0$. 
By Lemma \ref{lem:our 19.4 the construction has property A} $\Gamma$ is an $A$-map and so by Lemma \ref{lem:if Gamma is with minimal num of cells then r(Gamma)=0}, we have that $r(\Gamma)=0$.

Since $(\Pi,\Gamma,q)\geq \epsilon$, the label of $q$ has a subword $T$ equals in rank $0$, and so, %(by irreducibility,)
visually equal, to a subword of the label of $\partial \Pi$. Since the label of $\partial \Pi$ is a relator of rank $>0$, it is of the form $A^k$ for some word $A$ that is a simple word of rank $\geq 0$, and  $T$ has length $\geq \epsilon |A^k|$. Since $\epsilon k>2$, it follows that $T$ must contain $A$ as a subword. This is impossible, since by requirement (S3) of the definition of a simple element $A$ must contain a letter from $S\setminus \{a,b\}$, while $T$, as a subword of $\phi(q)$ contains only the letters $\{a,b\}$. 
\end{proof}

%Let $\|\cdot \|_G$ denote the length of an element in $G$ with respect to the generating set $S$. That is, $\|g\|$ is the length $|U|$ of a shortest word $U\in F(S)$ with $U$ representing $g$ in $G$.% and $\|\cdot\|_K$ the length of an element in $K$ with respect to the generating set $\{a,b\}^{\pm 1}$.

For an element $g\in G$ we define its norm $\|g\|$ to be the minimum length $|U|$ of a word $U\in F(S)$ representing $g$ in $G$. Similarly, for a word $K\in F(S)$, we denote by $\|K\|$ the norm of the image of $K$ in $G$. Evidently $\|K\|\leq |K|$.

%Consider the finite alphabet $S=\{c_1,\dots,c_m,a,b\}$, and the set $F(S)$ of all finite reduced words over $S$. Let $\calK=F(a,b)$ denote the subset of all finite reduced words over $\{a,b\}^{\pm1}$. Consider further the group $G$ as in our construction, with the graded presentation %\begin{equation}\label{eq:G=<S|bigcup R_i>}
%    G=\langle S \mid R=1, R\in \calR = \bigcup \calR_i\rangle
%\end{equation}

The next lemma formalizes the fact that the subgroup $\langle a,b\rangle$ forms a convex subset of the Cayley graph $\Gamma(G,S)$. In particular, this implies that $\langle a,b\rangle$ is a free subgroup, see the corollary thereafter. 

\begin{lemma}\label{lem:K is convex}
Let $K$ be a word in $F(a,b)$ and $U$ an arbitrary word in $F(S)$, such that $U=K$ in $G$ and $|U|\leq |K|$. Then $U\equiv K$.

In particular, for every word $K\in F(a,b)$ we have $\|K\|=|K|$.
\end{lemma}

\begin{proof}
Suppose not. Then there exist $K\in F(a,b)$ and $U\in F(S)$ such that $U=K$ in $G$, $|U|\leq |K|$, but $U\nequiv K$. Let $K,U$ be such words and such that the cyclic word $KU^{-1}$ is of minimal possible length. 

By Theorem \ref{lem:thm 13.1 van kampen} there exists a reduced graded circular diagram $\Delta$ over the graded presentation \ref{eq:G=<S|bigcup R_i>} with contour $p$ and label $\phi(p)=KU^{-1}$. The contour of $\Delta$ decomposes as $p=q^1q^2q^3q^4$ where $\phi(q^1)=K,\phi(q^2)=U^{-1},\phi(q^3)=\phi(q^4)=\emptyset$. We have that $q^1,q^2$ are smooth section by minimality of $K,U$, while $q^3,q^4$ are smooth sections by definition. Since $\Delta$ has only a finite number of cells, it is a diagram of rank $j$ for some $j\geq 0$ %\gil{isn't the reason it has finite cells is the fact it is of finite rank?}.

We claim that $r(\Delta)=0$. Indeed, if $r(\Delta)>0$, then by Lemma \ref{lem:our 19.4 the construction has property A}, $\Delta$ is an $A$-map. Let $\Pi$ be an arbitrary cell in $\Delta$. By Lemma \ref{lem:contiguity degree to <a,b> is small}, the contiguity degree of $\Pi$ to $p_K$ is less than $\epsilon$. By Lemma \ref{lem:contiguity degree to smooth section is always small}, the contiguity degree of $\Pi$ to $p_U$ is smaller than $\bar{\alpha}$. Since $\epsilon + \bar{\alpha}<\bar{\gamma}$, we get that $\Pi$ is not a $\gamma$-cell. But Lemma \ref{lem:annular Greendlinger-like lemma} ensures that a $\gamma$-cell exists, deriving a contradiction. We therefore get $r(\Delta)=0$. This implies that $KU^{-1}\equiv 1$ visually, and so $K\equiv U$, deriving a contradiction again, and completing the proof of the lemma.
\end{proof}

\begin{corollary}\label{lem:H is free}
The subgroup $\langle a,b \rangle\leq G$ is a non-abelian free group.
\end{corollary}

\begin{proof}
Let $K$ be a freely reduced word in $F(a,b)$ with positive length, $|K|>0$. By Lemma \ref{lem:K is convex}, $\|K\|=|K|>0$, and in particular, $K\neq 1$ in $G$.
\end{proof}

%We introduce another notation: for an element $g\in G$ which is conjugate to an element in the subgroup $\langle a,b\rangle$, we denote by $\|g\|_*$ the minimal length of a word $VAV^{-1}$ such that $A\inF(a,b)$, $V\in F(S)$ and $VAV^{-1}=g$ in $G$.

We next show that also the union of conjugates of $\langle a,b \rangle$ enjoys convexity qualities.

\begin{lemma}\label{lem:conjugates of K have length 2V+K}
Suppose an element $g\in G$ is conjugate to an element in the subgroup $\langle a,b\rangle$.
Then there exist $K\in F(a,b)$ and $V\in F(S)$ such that $VKV^{-1}$ represents $g$ in $G$ and such that $\|g\|=|VKV^{-1}|=2|V|+|K|$ %That is, such that $\|g\|_*=2|V|+|K|$.
%Let $U\in F(S)$ be a word representing $g$ in $G$. Then $|U|\geq |VKV^{-1}|=2|V|+|K|$.

%In particular, for every $g\in G$ that is conjugate to an element of $\langle a,b\rangle$ we have $\|g\|=2|V|+|K|$ where $K\in F(a,b),V\in F(S)$, and $VKV^{-1}$ represents $g$ in $G$. 
\end{lemma}

\begin{proof}
If not, let $g\in G$ be an element of minimal norm, for which the lemma is false, and let $U$ be a word representing $g$ in $G$, with $|U|=\|g\|$. By assumption, $g$ is conjugate to an element from $\langle a,b\rangle$. 
Let $K$ be a shortest word in $F(a,b)$ representing an element that is conjugate to $g$ in $G$.
Note that by minimality of the choice, both $U$ and $K$ are cyclically reduced, as otherwise, they could be replaced by a shorter conjugate. We can furthermore assume that $U\notin F(a,b)$, as the case $U\in F(a,b)$ is taken care of in Lemma \ref{lem:K is convex}.
%From Lemma \ref{lem:K is convex} we have that $|K|=\|k\|$.

By Lemma \ref{lem:thm 13.2 annular van-kampen} there exists a reduced graded annular diagram $\Delta$ over Presentation \ref{eq:G=<S|bigcup R_i>} with contours $p_K$ and $p_U$ and labels $\phi(p_K)=K$, $\phi(p_U)=U^{-1}$. Moreover, the paths $p_U$ and $p_K$ by minimality of the choice.

We claim that $r(\Delta)=0$. Indeed, suppose that $r(\Delta)>0$ and let $\Pi$ be an arbitrary cell in $\Delta$. By Lemma \ref{lem:contiguity degree to <a,b> is small}, the contiguity degree of $\Pi$ to $p_K$ across an arbitrary contiguity submap is less than $\epsilon$. By Lemma \ref{lem:our 19.4 the construction has property A} $\Delta$ is an $A$-map, and so by Lemma \ref{lem:contiguity degree to smooth section is always small}, the contiguity degree of $\Pi$ to $p_U$ across an arbitrary contiguity submap is less than $\bar{\alpha}$. Since $\epsilon + \bar{\alpha}<\bar{\gamma}$, $\Pi$ is not a $\gamma$-cell. But Lemma \ref{lem:annular Greendlinger-like lemma} ensures that a $\gamma$-cell exists, deriving a contradiction.

The rank of $\Delta$ being $0$ means that $\Delta$ is a diagram over the free group. Using Lemma \ref{lem:thm 13.2 annular van-kampen} again, $K$ and $U$ are conjugates over the free group.
Since $U\notin F(a,b)$, it cannot be that $K\equiv VUV^{-1}$ for any word $V\in F(S)$. It follows that $VKV^{-1}\equiv U$ visually, and so $\|g\|=|U|=|VKV^{-1}|=2|V|+|K|$. We derived a contradiction, concluding the proof of the lemma. 
\end{proof}

%Recall that by Lemma \ref{lem:torsion}, each element in $G$ is either torsion, or conjugate to an element of the subgroup $\langle a,b\rangle$. We now use this, together with Lemma \ref{lem:conjugates of K have length 2V+K}, to compute the density of torsion elements in $G$.

\section{The law $x^k=1$ holds almost surely}\label{sec:proof of main theorem}

As in \cite{OL-book}, let $k$ be a sufficiently large odd number and let $G$ be as in Presentation \ref{eq:G=<S|bigcup R_i>} ($m$ will be specified in the sequel).

Let $H$ denote the subgroup $H=\langle a,b\rangle$. For an element $h\in H$, denote by $\|h\|_H$ its norm with respect to the generating set $\{a^{\pm 1},b^{\pm 1}\}$. As before, the notation $\|g\|$ for $g\in G$ is used to denote the norm of $g$ with respect to the generating set $S\cup S^{-1}$.

For $r\in \mathbb{R}_{\geq 0}$ let: $$ B_G(r)=\{g\in G\ |\ \|g\|\leq r\},\ B_H(r)=\{h\in H\ |\ \|h\|_H\leq r\} $$
be the corresponding $r$-balls in $G,H$ respectively and let
$$ \gamma_G(n) = \# B_G(n),\ \gamma_H(n) = \# B_H(n) $$
denote the corresponding growth functions.

Recall that the growth exponent $\alpha:=\lim_{n\rightarrow \infty} \sqrt[n]{\gamma_G(n)}$ exists by Fekete's lemma. 
By \cite[Theorem~1.3]{Coulon_exp}, the free Burnside group $B(m,k)$ has exponential growth and moreover, its growth exponent (with respect to the standard set of generators) is at least $\alpha_{m,k}$, where $\alpha_{m,k}\xrightarrow{m\rightarrow \infty} \infty$.
Therefore $G$ grows exponentially, and since $G\twoheadrightarrow B(m,k)$, the exponent $\alpha$ is bounded from below by $\alpha_{m,k}$. Let us pick $m\gg 1$ such that $\alpha_{m,k}$ (and hence $\alpha$) is greater than $7$. Let $N$ be such that for $r\geq N$, $$ (\alpha-1)^r\leq \gamma_G(r)\leq (\alpha+1)^r. $$ In addition, recall that $\gamma_H(r)\leq C\cdot 3^r$ for some $C>0$, the latter bounding the growth function of a free group on two generators with respect to the standard generators.
% , Lemma \ref{lem:H is free}.
Finally, let $H^G := \bigcup_{x\in G} xHx^{-1}$ denote the union of conjugates of $H$ and let $\langle\langle H \rangle\rangle := \bigcap_{\substack{N\trianglelefteq G \\ H\leq N}} N$ denote the normal closure of $H$ in $G$.

We now compute the density of torsion elements in $G$. By Lemma \ref{lem:torsion}, this amounts to computing the density of $H^G$.

\begin{lemma} \label{lem:uniform}
We have:
$$ \frac{\#\left(B_G(n)\cap H^G\right)}{\#B_G(n)}\xrightarrow{n\rightarrow \infty} 0. $$
\end{lemma}
\begin{proof}
By Lemma \ref{lem:conjugates of K have length 2V+K},
$$ B_G(n)\cap H^G = \bigcup_{i=0}^{n} \Big\{xax^{-1}\ |\ a\in B_H(i),\ x\in B_G\left(\frac{n-i}{2}\right)\Big\}.$$
Therefore:
\begin{eqnarray*}
\#\left(B_G(n)\cap H^G\right) & \leq & \sum_{i=0}^{n} \gamma_H(i)\gamma_G\left(\frac{n-i}{2}\right) \\
& = & \sum_{i=0}^{n-2N} \gamma_H(i)\gamma_G\left(\frac{n-i}{2}\right) + \sum_{i=n-2N+1}^{n} \gamma_H(i)\gamma_G\left(\frac{n-i}{2}\right) \\
& \leq & \sum_{i=0}^{n-2N} \gamma_H(i)\gamma_G\left(\frac{n-i}{2}\right) + 2N\cdot C3^n \cdot \gamma_G(N)
\end{eqnarray*}
Now denote $D:=2NC \gamma_G(N)$ (this is a constant that does not depend on $n$) and notice that for $i\leq n-2N$ we have $\frac{n-i}{2}\geq N$ so $\gamma_G\left(\frac{n-i}{2}\right)\leq (\alpha+1)^{\frac{n-i}{2}}$.
Therefore:
\begin{eqnarray*}
\#\left(B_G(n)\cap H^G\right) & \leq & \sum_{i=0}^{n-2N} C\cdot 3^i\cdot (\alpha+1)^{\frac{n-i}{2}} + D \cdot 3^n \\
& \leq & C\sum_{i=0}^{n} {n \choose i}\cdot 3^i\cdot \sqrt{\alpha+1}^{n-i} + D \cdot 3^n \\
& \leq & C\cdot \left(3 + \sqrt{\alpha+1}\right)^n + D\cdot 3^n \\
& \leq & C' \left(3 + \sqrt{\alpha+1}\right)^n
% C\cdot (3+\sqrt{\alpha+1})^n+D\cdot 3^n \\ & \leq & C'\cdot (3+\sqrt{\alpha+1})^n
\end{eqnarray*}
for some new constant $C'>0$. Since $\alpha>7$ it follows that $3 + \sqrt{\alpha+1}<\alpha-1$ and so:
$$ \frac{\#\left(B_G(n)\cap H^G\right)}{\#B_G(n)} \leq \frac{C' \cdot\left(3 + \sqrt{\alpha+1}\right)^n}{(\alpha-1)^n}\xrightarrow{n\rightarrow \infty} 0, $$
as claimed.
\end{proof}

We conclude by observing that $G$ satisfies the requirements of Theorem \ref{thm:main_construction}.

\begin{proof}[{Proof of Theorem \ref{thm:main_construction}}]

We now show that the group $G$, defined in \ref{eq:G=<S|bigcup R_i>} satisfies the requirements of the theorem. As before, we denote by $H$ the subgroup $H=\langle a,b\rangle$.
By Lemma \ref{lem:H is free}, $H\leq G$ is a free non-abelian subgroup, so Item (3) is satisfied.

Let %$\mu_n=\mathcal{U}(B_{G,S}(n))$ 
$\mu_n=U_{B_{G}(n)}$ be the uniform distribution on the $n$-th ball of the Cayley graph $\Gamma(G,S)$. Then:
$$ \Pr_{\mu_n} (x^k = 1) = \frac{\#\left(B_G(n)\cap H^G\right)}{\#B_G(n)} \xrightarrow{n\rightarrow \infty} 0 $$
by Lemma \ref{lem:uniform}, proving Item (1).

To prove Item (2), let $\nu$ be a finitely supported step-distribution on $G$ such that $\text{supp}(\nu)$ generates $G$ as a semigroup, and consider the sequence of measures $\{\nu^{*n}\}_{n=1}^{\infty}$, which represents the corresponding random walk on $G$.

Recall that by Corollary \ref{cor:quotient burnside}, $G/\langle\langle H \rangle\rangle\cong B(m,k)$; denote the resulting surjection by $\pi\colon G\twoheadrightarrow B(m,k)$ and let $\pi_\sharp \nu$ be the pushforward measure. Since by Lemma \ref{lem:torsion} every element in $G$ which does not have order dividing $k$ is conjugate to an element from $H$, we obtain:

\begin{eqnarray*}
\Pr_{\nu^{*n}} (x^k \neq 1) & = & \Pr_{\nu^{*n}}\left(x\in H^G \setminus \{1\} \right) \\ & \leq & \Pr_{\nu^{*n}}\left(x\in \langle\langle H \rangle\rangle\right) \\
& = & \Pr_{(\pi_\sharp \nu)^{*n}}(x=1)
\end{eqnarray*}

Now $\Pr_{(\pi_\sharp \nu)^{*n}}(x=1)$ is the return probability of the random walk on $G/\langle\langle H \rangle\rangle\cong B(m,k)$ with step-distribution $\pi_\sharp \nu$, whose support is a finite generating set of $B(m,k)$. 
Therefore $\Pr_{(\pi_\sharp \nu)^{*n}}(x=1)\xrightarrow{n\rightarrow \infty} 0$ and hence:
$$ \Pr_{\nu^{*n}} (x^k \neq 1) \xrightarrow{n\rightarrow \infty} 0,$$
as required.
\end{proof}

\appendix
\section{$A$-maps over Presentation \ref{eq:G=<S|bigcup R_i>}} %More than tentative title

Condition $A$ plays a crucial role in verifying the desired properties of our construction. The main objective of this appendix is to prove Lemma \ref{lem:lem 19.4}, namely, that every map over Presentation \ref{eq:G=<S|bigcup R_i>} satisfies Condition $A$. The proof of this lemma goes step by step as in the original proof of \cite[Lemma 19.4]{OL-book}. It is done by a simultaneous proof of a sequence of lemmas by induction on the rank of the map. This allows the proof to use inner references to other lemmas: by the induction hypothesis, every rank-$i$ map over \ref{eq:G=<S|bigcup R_i>} is an $A$-map, and so the other lemmas hold for rank $i$ maps as well, allowing the proof for maps of rank $i+1$. 
% The other lemmas included in this appendix appear for completeness reasons, as they play a role in the verification of the proof of \ref{lem:lem 19.4}.

This appendix also provides a road-map for the reader who wishes to observe the similarities and differences between the geometric properties of Presentation \ref{eq:G=<S|bigcup R_i>} established in the current article and Ol'shanskii's presentation of a free Burnside group constructed in \cite[§18]{OL-book}.

Since the presentations are indeed similar in spirit, many of the claims from \cite[§18,§19]{OL-book}, which describe properties of Ol'shanskii's presentation and maps over it, hold also for Presentation \ref{eq:G=<S|bigcup R_i>}.
Although in most cases this is a word-to-word repetition of \cite{OL-book}, we chose to re-state the statements here, to emphasize that despite the difference in the definition of `simple elements' and `periods', the statements hold for Presentation \ref{eq:G=<S|bigcup R_i>} as well.

We distinguish between three types of lemmas, according to the level of similarity between our presentation and Olshankii's:
\begin{enumerate}[label=(\Roman*)]
    \item Lemmas that hold `as is' in our case, and are brought here together with an overview of tools and inner references utilized in their proofs.
    \item Lemmas whose statements hold `as is' in our case but whose proofs require some minor adaptations, due to the aforementioned differences in the constructions; these are brought here, accompanied with explanations of the required adaptation of the proofs.
    \item Lemmas whose statements in our case are different than their parallels in \cite{OL-book}; these are brought here with complete proofs, although these proofs are often quite similar to the proofs of their parallels from \cite{OL-book}.
\end{enumerate}

We now list the lemmas appearing in this appendix, according to the above types, together with their counterparts from \cite{OL-book} and the lemmas used to prove them by simultaneous induction.

\begin{center}
\begin{tabular}{||c c c c||} 
 \hline
 Lemma & Counterpart from \cite{OL-book} & Type & Inner dependencies \\ [0.5ex] 
 \hline\hline
 \ref{lem:ours 18.1} & Lemma 18.1 & III & \\
 \hline
 \ref{lem:ours 18.3} & Lemma 18.3 & II & \ref{lem:ours 18.1},\ref{lem:lem 19.4},\ref{lem:19.5 section of reduced diagram with periodic contour is smooth},Lemma \ref{lem:thm 13.1 van kampen},\cite[Lemma 17.1]{OL-book} \\
 \hline
 \ref{lem:lem 18.4} & Lemma 18.4 & I & \ref{lem:lem 19.4},\ref{lem:19.5 section of reduced diagram with periodic contour is smooth} \\
 \hline
 \ref{lem:lem 18.5} & Lemma 18.5 & I & \ref{lem:lem 19.4} \\
 \hline
 \ref{lem:our lem 18.6} & Lemma 18.6 & II & \ref{lem:ours 18.1},\ref{lem:lem 18.5},\ref{lem:lem 18.8},\ref{lem:our lem 19.3},\ref{lem:lem 19.4},\ref{lem:19.5 section of reduced diagram with periodic contour is smooth},Lemma \ref{lem:K is convex} \\
 \hline
 \ref{lem:lem 18.7} & Lemma 18.7 & I & \ref{lem:our lem 18.6},\ref{lem:lem 19.4},\ref{lem:19.5 section of reduced diagram with periodic contour is smooth} \\
 \hline
 \ref{lem:lem 18.8} & Lemma 18.8 & I & \ref{lem:lem 18.4},\ref{lem:lem 18.7},\ref{lem:lem 19.4},\ref{lem:19.5 section of reduced diagram with periodic contour is smooth} \\
 \hline
 \ref{lem:lem 18.9} & Lemma 18.9 & I & \ref{lem:ours 18.3} \\
 \hline
 \ref{lem:our lem 19.1} & Lemma 19.1 & I & \ref{lem:lem 18.4},\ref{lem:lem 18.7},\ref{lem:lem 19.4},\ref{lem:19.5 section of reduced diagram with periodic contour is smooth} \\
 \hline
 \ref{lem:19.2 either q1 is small, or q1,q2 are compatible} & Lemma 19.2 & I & \ref{lem:our lem 18.6},\ref{lem:lem 18.9},\ref{lem:our lem 19.1},\ref{lem:lem 19.4},\ref{lem:19.5 section of reduced diagram with periodic contour is smooth} \\
 \hline
 \ref{lem:our lem 19.3} & Lemma 19.3 & III & \\
 \hline
 \ref{lem:lem 19.4} & Lemma 19.4 & I & \ref{lem:19.2 either q1 is small, or q1,q2 are compatible},\ref{lem:our lem 19.3},\ref{lem:19.5 section of reduced diagram with periodic contour is smooth} \\
 \hline
 \ref{lem:19.5 section of reduced diagram with periodic contour is smooth} & Lemma 19.5 & I & \ref{lem:19.2 either q1 is small, or q1,q2 are compatible},\ref{lem:our lem 19.3} \\
 \hline
 \ref{lem:13.2 boundary of two cells can't be compatible} & Lemma 13.2 & I & \\ \hline
\end{tabular}
\end{center}

\bigskip

The next lemmas use small positive constants $\alpha>\beta>\gamma>\epsilon>\zeta$ defined as in \cite[§15.1]{OL-book} ($\alpha,\gamma,\epsilon$ has already appeared above).

\begin{lemma}[{cf. \cite[Lemma 18.1]{OL-book}}]\label{lem:ours 18.1}
    Every word $X$ is conjugate in rank $i\geq 0$ to either a power of a period of rank $j\leq i$, an element of $\langle a,b\rangle\leq G(i)$, or a power of a word simple in rank $i$.
\end{lemma}
\begin{proof}
    Fix $i\geq 0$. We induct on $|X|$. Suppose first that $X$ is not conjugate in $G(i)$ to a power of a shorter word. If $X$ is not conjugate to a power of a period of rank $j\leq i$, and furthermore, $X$ is not conjugate in $G(i)$ to an element of $\langle a,b\rangle\leq G(i)$, then it is simple in rank $i$ by definition. So the assertion of the lemma holds in this case.

    We can therefore assume that $X=ZY^lZ^{-1}$ in $G(i)$ for some word $Y$ strictly shorter than $X$. By the inductive hypothesis, $Y=UA^mU^{-1}$, where $A$ is either a period of rank $j\leq i$, an element of $\langle a,b\rangle\leq G(i)$, or simple in rank $i$. Then $X=(ZU)A^{ml}(ZU)^{-1}$ completes the proof.
\end{proof}

%A slight problem, is that the following lemmas are built one over the other. To avoid forward references we probably need to prove them all at once, by induction. Or, prove each for every $i$ separately.

\begin{lemma}[{cf. \cite[Lemma 18.3]{OL-book}}]\label{lem:ours 18.3}
If $X\neq 1$ and $X$ has finite order in rank $i$, then it is conjugate in rank $i$ to a power of some period of rank $k\leq i$.
\end{lemma}
\begin{proof}
    Suppose not. Then by Lemma \ref{lem:ours 18.1}, $X$ is either conjugate to an element from $\langle a,b\rangle \leq G(i)$, or to a power of a word simple in rank $i$.
    The first case is impossible by Lemma \ref{lem:H is free}.%\footnote{We use simultaneous induction here}. 
    In the second case, we obtain a word $A$ simple in rank $i$, which has finite order, $A^s=1$, $s\neq 0$. By Lemma \ref{lem:lem 19.4} and Lemma \ref{lem:19.5 section of reduced diagram with periodic contour is smooth}, a reduced diagram for this equation (which exists by Lemma \ref{lem:thm 13.1 van kampen}) is an $A$-map, and its contour $Q$ is a smooth section, which is impossible by \cite[Lemma 17.1]{OL-book}, with $|t|=0,|q|>0$.
\end{proof}

\begin{lemma}[{\cite[Lemma 18.4]{OL-book}}]\label{lem:lem 18.4}
     If $A$ and $B$ are simple in rank $i$ and $A =XB^lX^{-1}$ for some $X$, then $l=1$.
\end{lemma}

\begin{proof}
    The proof is identical to the original proof in \cite{OL-book}. It uses Lemma \ref{lem:lem 19.4} %(Lemma 19.4 in \cite{OL-book}) 
    and Lemma \ref{lem:19.5 section of reduced diagram with periodic contour is smooth}. %(Lemma 19.5 in \cite{OL-book}).
    It also uses the fact that if two words $A$ and $B$ that are simple in rank $i$ satisfy that $A=XB^lX^{-1}$ in rank $i$ then $|A|\leq |B|$, which indeed holds for Presentation \ref{eq:G=<S|bigcup R_i>} by Item (S2) of the definition of a simple element.
%By definition of simple in rank $i$, $A$ is not conjugate in rank $i$ to a power of a shorter word, so we must have $|A|\leq |B|$. By Lemma \ref{lem:thm 13.2 annular van-kampen}, there is a reduced annular diagram $\Delta$ of rank $i$ with contours $p$ and $q$, where $\phi(p)\equiv A$ and $\phi(q)\equiv B^{-l}$. By Lemma \ref{lem:19.5 section of reduced diagram with periodic contour is smooth}, its contours are smooth cyclic sections, and by Lemma \ref{lem:our 19.4 the construction has property A}, $\Delta$ is an $A$-map. Then, by Lemma \ref{lem:thm17.1 if qt is contour of map then (1-beta)q<t}, $\bar{\beta}|B^l|<|A|$, that is, $l\bar{\beta}|B|<|A|$, which for $l\geq 2$ contradicts the inequality $|A|\leq |B|$, since $2\bar{\beta}=2-2\beta>1$.
\end{proof}

\begin{lemma}[{\cite[Lemma 18.5]{OL-book}}]\label{lem:lem 18.5}
If words $X$ and $Y$ are conjugate in rank $i$, then there is a
word $Z$ such that $X=ZYZ^{-1}$ and $|Z|\leq \bar{\alpha}(|X|+|Y|)$.
\end{lemma}
\begin{proof}
    The proof is identical to the original proof in \cite{OL-book}. It is done by simultaneous induction, using Lemma \ref{lem:lem 19.4}. %(Lemma 19.4 in \cite{OL-book}).
\end{proof}

\begin{lemma}[{\cite[Lemma 18.6]{OL-book}}]\label{lem:our lem 18.6}
   Let $\Delta$ be a reduced circular diagram of rank $i$ with contour $p_1q_1p_2q_2$. where $\phi(q_1)$ and $\phi(q_2)^{-1}$ are periodic words with period $A$ simple in rank $i$. If $|p_1|,|p_2|<\alpha|A|$ and $|q_1|,|q_2|>(\frac{5}{6}h+1)|A|$, then $q_1$ and $q_2$ are $A$-compatible in $\Delta$. (The inductive parameter is $L=|A|+|A|$.)
\end{lemma}
\begin{proof}
    The proof of Lemma 18.6 in \cite{OL-book} consists of five steps. The proof here is almost identical, with only a slight adaptation on Step 3, which is the following.

        %Step 1 uses nothing (just calculations). Step 2 uses \ref{lem:lem 19.4}, \ref{lem:19.5 section of reduced diagram with periodic contour is smooth}. A small gap in Step 3 is brought below. Step 4 uses nothing again (calculations). Lastly step 5 uses 18.8 (our \ref{lem:lem 18.8}) and the fact that a period of rank $i$ is simple in any rank $j<i$. and also the fact that if $A$,$B$ are conjugate in rank $i$ then it cannot be that $A$ is simple in rank $i$ and $|B|<|A|$.    
    
    %Since the only step that involves our adaptations is Step (3), we only repeat this one here. We can assume that $\phi(q_2)^{-1}$ begins with $A$ (replacing $A$ by a cyclic conjugate) and $\phi(q_1)$ with $A'A$, where $|A'|\leq \frac{1}{2}|A|$ (If $|A'|>\frac{1}{2}|A|$, then we assume conversely that $\phi(q_1)$ begins with $A$ and $\phi(q_2)^{-1}$ with $A'A$, giving $|A'|<\frac{1}{2}|A|$. By step (1) (which holds as is), we can assume that $\phi(p_1)\neq 1$ in rank $i$.
    
    %Without loss of generality, we can assume that $|q_1|\geq |q_2|$. Since the $A$-periodic words $\phi(q_1)$ and $\phi(q_2)^{-1}$ begin with $A$, it follows that $\phi(q_1)\equiv \phi(q_2)\phi(q')$.
    %By step (2), we then have $|q'|=|q_1|-|q_2|<(\bar{\beta}^{-1}-1)|q_2|+|A|$.

   By Lemma \ref{lem:ours 18.1} we have that $D=YB^mY^{-1}$ in $G(i)$\footnote{$D$ is as constructed in the original proof of Lemma 18.6 \cite{OL-book}}, where $m\neq 0$ and $B$ is either (1) simple in rank $i$; (2) period of rank $l\leq i$; or (3) an element of $\langle a,b\rangle\leq G(i)$. In cases (1) and (2), Lemma \ref{lem:lem 19.4} and Lemma \ref{lem:19.5 section of reduced diagram with periodic contour is smooth} %(that is, Lemma 19.4 and Lemma 19.5 in \cite{OL-book}) 
   imply that the cyclic section with label $B^m$ is smooth. In case (3), the cyclic section with label $B^m$ is smooth by Lemma \ref{lem:K is convex}. Then the proof proceeds as is.

    %Then step (4) holds as is.%, giving that $|v_2|>\bar{\beta}|v_1|-|x_1'|-|x_2'|>\frac{3}{4}\bar{\beta}h|A|-|A|>\frac{2}{3}h|A|>h|B|$.

    %Lastly, in step (5), we only use the fact that a period of rank $i$ is simple in any rank $j<i$. And the fact that if $A$ is simple in rank $i$, then it cannot be conjugate in rank $i$ to a word $B$ with $|B|<|A|$.

    We comment that Steps 1, 2, 4, and 5 of the proof use references to Lemmas \ref{lem:lem 18.5}, \ref{lem:lem 18.8}, \ref{lem:lem 19.4}, and \ref{lem:19.5 section of reduced diagram with periodic contour is smooth} %(that is, Lemmas 18.5, 18.8, 19.4, and 19.5 in \cite{OL-book}), 
    and
    Step 5 uses the fact that a period of rank $i$ is simple in any rank $j<i$ (Item \ref{Item:period of rank i simple in any lower rank} of Lemma \ref{lem:our lem 19.3}), and the fact that a word $A$ simple in rank $i$ cannot be conjugate to a shorter word, which is Item (S2) in Definition \ref{def:simple}.
    %We note that Step (1) does not count on any further knowledge on the construction and so it works as is.    Step (2) is only based on Lemma \ref{lem:thm17.1 if qt is contour of map then (1-beta)q<t}, which does not depend on the construction, and on Lemmas \ref{lem:our 19.4 the construction has property A}, and \ref{lem:19.5 section of reduced diagram with periodic contour is smooth}.
    %Step (3) uses Lemma \ref{lem:ours 18.1}, which is slightly different in our case. it is therefore important to verify this step!! \gil{verify here}.
    %Step (4) again uses no knowledge about the presentation. %Finally, Step (5) uses only Lemma \ref{lem:lem 18.8}, which holds here as is.
\end{proof}

%\gil{add here explanation: despite the seemingly self-dependence between 4.9 and the current lemma, the proof is coherent by simultaneous induction.}

\begin{lemma}[{\cite[Lemma 18.7]{OL-book}}]\label{lem:lem 18.7}
    Let $Z_1A^{m_1}Z_2=A^{m_2}$ in rank $i$, $m=\min(m_1,m_2)$, where $A$ is simple in rank $i$.
    If $|Z_1|+|Z_2|<(\gamma(m-\frac{5}{6}h-1)-1)|A|$, then $Z_1$ and $Z_2$ are equal in rank $i$ to powers of $A$. (The inductive parameter is $L = 2|A|$.)
\end{lemma}
\begin{proof}
    The proof is identical to the original proof in \cite{OL-book}. It uses Lemma \ref{lem:our lem 18.6}, Lemma \ref{lem:lem 19.4}, and Lemma \ref{lem:19.5 section of reduced diagram with periodic contour is smooth}. %(Lemmas 18.6, 19.4, and 19.5 in \cite{OL-book}).
\end{proof}

\begin{lemma}[{\cite[Lemma 18.8]{OL-book}}]\label{lem:lem 18.8}
Let $\Delta$ be a reduced circular diagram of rank $i$ with contour $p_1q_1p_2q_2$ where $\phi(q_1),\phi(q_2)$ are periodic words with periods $A,B$ simple in rank $i$ with $|A|\geq |B|$. If $|p_1|,|p_2|<\alpha |B|$, $|q_1|>\frac{3}{4}h|A|,|q_2|>h|B|$, then $A$ is conjugate in rank $i$ to $B^{\pm1}$. Moreover, if $\phi(q_1)$ and $\phi(q_2)^{-1}$ begin with $A$ and $B^{-1}$, respectively, then $A=\phi(p_1)^{-1}B^{\pm 1}\phi(p_1)$. (The inductive parameter is $L=|A|+|B|$.) 
\end{lemma}

\begin{proof}
        The proof is identical to the original proof in \cite{OL-book}. It uses Lemma \ref{lem:lem 18.4}, Lemma \ref{lem:lem 18.7}, Lemma \ref{lem:lem 19.4}, and Lemma \ref{lem:19.5 section of reduced diagram with periodic contour is smooth}. %(Lemmas 18.4, 18.5, 19.4, and 19.5 in \cite{OL-book}).
\end{proof}

\begin{lemma}[{\cite[Lemma 18.9]{OL-book}}]\label{lem:lem 18.9}
    Let $\Delta$ be a reduced circular diagram of rank $i$ with contour $p_1q_1p_2q_2$ where $\phi(q_1)$ and $\phi(q_2)$ are periodic words with period $A$ simple in rank $i$, and $\max(|p_1|,|p_2|)<\alpha |A|$. Then $\max(|q_1|,|q_2|)\leq h|A|$.
\end{lemma}
\begin{proof}
    The proof is identical to the original proof in \cite{OL-book}. It uses Lemma \ref{lem:ours 18.3}. %(Lemma 18.3 in \cite{OL-book}). 
    We also comment that oddness of the exponent $k$ is used here.
\end{proof}

%\subsection*{Section §19}

% Next we state and sort the relevant lemmas from \cite{OL-book}, Section §19.  The only lemma here whose statement is different than in \cite{OL-book} is Lemma  \ref{lem:our lem 19.3} (to compare with Lemma 19.3 in \cite{OL-book}). We provide here its full proof. We comment that Items (2) and (3) of the statement of the lemma are identical to those of Lemma 19.3 in \cite{OL-book}. As a result, any other proof that relies on these two items, is yet valid as is. Lemmas \ref{lem:our lem 19.1}, \ref{lem:19.2 either q1 is small, or q1,q2 are compatible}, \ref{lem:lem 19.4}, and \ref{lem:19.5 section of reduced diagram with periodic contour is smooth} (namely, 19.1, 19.2, 19.4, and 19.5 in \cite{OL-book}) hold here with the identical statements and proofs. We include here only the statements, together with a short comment regarding the tools and references used in the proofs.

\begin{lemma}[{\cite[Lemma 19.1]{OL-book}}]\label{lem:our lem 19.1}
Let $\Delta$ be a reduced circular diagram of rank $i$ with contour $p_1q_1p_2q_2$,where $\phi(q_1)$, $\phi(q_2)$ are periodic words with periods $A,B$ simple in rank $i$, and suppose that:
$$|p_1|,|p_2|<\alpha|B|,\ \ |q_1|>(1+\frac{1}{2}\gamma)|A|,\ \ |q_2|>\frac{1}{2}\epsilon k|B|.$$
Then $A$ is conjugate in rank $i$ to $B^{\pm1}$, and if $\phi(q_1)$ begins with $A$ and $\phi(q_2^{-1})$ with $B^{-1}$, then $A=\phi(p_1)^{-1}B^{\pm1}\phi(p_1)$ in $G(i)$.
\end{lemma}

\begin{proof}
    The proof is identical to the original proof in \cite{OL-book}. It uses Lemmas \ref{lem:lem 18.4}, \ref{lem:lem 18.7}, \ref{lem:lem 19.4}, and \ref{lem:19.5 section of reduced diagram with periodic contour is smooth}. %(That is, Lemmas 18.4, 18,7, 19.4, and 19.5 from \cite{OL-book}).
\end{proof}

\begin{lemma}[{\cite[Lemma 19.2]{OL-book}}]\label{lem:19.2 either q1 is small, or q1,q2 are compatible}
    Let $\Delta$ be a reduced circular diagram of rank $i$ with contour $p_1q_1p_2q_2$ where $\phi(q_1),\phi(q_2)$ are periodic words with periods $A,B$ simple in rank $i$. and $|q_2|\geq \epsilon k |B|$ and $\max(|p_1|,|p_2|)<\zeta kc$ with $c=\min(|A|,|B|)$. Then either $|q_1|<(1+\gamma)|A|$ or $A$ is conjugate in rank $i$ to $B^{\pm 1}$, and if $A\equiv B^{\pm1}$, then $A 
    \equiv B^{-1}$ and $q_1$ and $q_2$ are $A$-compatible in $\Delta$.
\end{lemma}

\begin{proof}
 The proof is identical to the original proof in \cite{OL-book}. It uses Lemmas \ref{lem:our lem 18.6}, \ref{lem:lem 18.9}, \ref{lem:our lem 19.1}, \ref{lem:lem 19.4}, and \ref{lem:19.5 section of reduced diagram with periodic contour is smooth}. %(Lemmas 18.6, 18.9, 19.1, 19.4, and 19.5 from \cite{OL-book}).
\end{proof}

The next lemma is similar to Lemma 19.3 in \cite{OL-book}, but here adapted to our construction.

\begin{lemma}[{cf. \cite[Lemma 19.3]{OL-book}}]\label{lem:our lem 19.3}
    \begin{enumerate}
        %\item For any word $X$, if $|X|\leq i$, then $X^k=1$.
        \item For any word $X$, if $|X|\leq i$, then either $X^k=1$ in $G(i)$, or $X$ is conjugate in $G(i)$ to a word from $\langle a,b \rangle$.
        \item A period of rank $i$ is simple in any rank $l<i$.\label{Item:period of rank i simple in any lower rank}
        \item If two periods $A,B^{\pm 1}$ of ranks $i,j$ are conjugate in rank $l<\min(i,j)$ then $A\equiv B$.\label{Item:conjugate periods are equiv}
    \end{enumerate}
\end{lemma}

\begin{proof}
\begin{enumerate}
    \item
    If $X$ is simple in rank $i-1$, then by construction $X$ is conjugate in rank $i-1$ to $A^{\pm 1}$ for some $A\in \calX_i$, and it follows that $X^k=1$ in $G(i)$. If $X$ is not simple in rank $i-1$, then it is conjugate in $G(i-1)$ to either a power $B^m$ of a period $B$ of rank $l\leq i-1$, to an element $K$ of $\langle a,b \rangle$, or to a power of a word $C$ with $|C|<|X|$. In the first case, $B^k=1$ in $G(i-1)$ implies $X^k=1$ in $G(i-1)$ and so also in $G(i)$. In the second case, $X$ and $K$ are conjugate in $G(i-1)$ and therefore in $G(i)$ so the assertion follows. We deal with the third case by induction on the length of the word. Since $|C|<|X|$, the induction hypothesis asserts that either $C^k=1$ in $G(i)$, or $C$ is conjugate in $G(i)$ to a word from $\langle a,b \rangle$. In both cases, the assertion follows since $X$ is conjugate to $C$.
    
 %   If $X$ is not simple in rank $i-1$, then by Lemma \ref{lem:ours 18.3}, $A$ does not have a finite order in any rank $l\leq i$, since otherwise this would imply that it is conjugate in $G(l)$ to a power of a period of rank $j\leq l\leq i$, contradicting the fact it is simple in rank $i$.
    \item %Note that $G(l')$ projects onto $G(l)$ for any $l'\leq l$, by construction. Since the definition of 'simple in rank $l$' is based on the relations that occur $G(l)$, it follows that every word simple in rank $l'$ is also simple in rank $l$.  
    Let $A$ be a period of rank $i$. By definition, $A$ is simple in $G(i-1)$, and it follows that $A$ is simple in rank $l$ for all $l\leq i-1$, since in this case the group $G(l)$ projects on $G(i-1)$.
    \item Without loss of generality assume $i\leq j$. If $i<j$, then $B^{\pm 1}$ is conjugate in $G(i)$ to the period $A$, and so it is not simple in rank $i$. But this contradicts Item (2) above. If $i=j$, then $A\equiv B$ by definition of the set $\calX_i$.
    \end{enumerate}
\end{proof}
Note the similarity of Item (1) of Lemma \ref{lem:our lem 19.3} to Lemma \ref{lem:torsion}.

We now prove Lemma \ref{lem:our 19.4 the construction has property A}, stated in Section \ref{sec:condition A}. The rank in the statement is $i+1$ (rather than $i$) to emphasize the use of a simultaneous induction in the proof.
\begin{lemma}[{Lemma \ref{lem:our 19.4 the construction has property A}, analogous to \cite[Lemma 19.4]{OL-book}}]\label{lem:lem 19.4}
    A reduced diagram $\Delta$ of rank $i+1$ is an $A$-map.
\end{lemma}
\begin{proof}
The proof is identical to the original proof in \cite{OL-book}.  

The verification of $(A1)$ (from Definition \ref{def_A}) uses only the fact that any rank $j$ relator has length $kj$, which holds here by construction.
The proof of $(A2)$ uses the fact that periods of rank $j\leq i+1$ are cyclically irreducible, and the fact that a word $A$ that is simple in rank $j-1$ cannot be conjugate in rank $j-1$ to a shorter word, which is Item (S2) Definition \ref{def:simple}. It also uses Lemma \ref{lem:19.5 section of reduced diagram with periodic contour is smooth}. %(Lemma 19.5 in \cite{OL-book}).
Lastly, the proof of $(A3)$ uses Lemma \ref{lem:19.2 either q1 is small, or q1,q2 are compatible} %(Lemma 19.2 in \cite{OL-book}) 
and Items \ref{Item:period of rank i simple in any lower rank} and \ref{Item:conjugate periods are equiv} of Lemma \ref{lem:our lem 19.3}. %(Lemma 19.3 in \cite{OL-book}).
\end{proof}

\begin{lemma}[{\cite[Lemma 19.5]{OL-book}}]\label{lem:19.5 section of reduced diagram with periodic contour is smooth}
Let $p$ be a section of the contour of a reduced diagram $\Delta$ of rank $i+1$ whose label is $A$-periodic, where $A$ is simple of rank $i+1$ or is a period of rank $l\leq i+ 1$, and in the latter case $\Delta$ has no cells of rank $l$ $A$-compatible with $q$. (If $p$ is a cyclic section, then we further require that $\phi(p)\equiv A^m$ for some integer $m$.) Then $p$ is a smooth section in the contour of $\Delta$.
\end{lemma}

\begin{proof}
The proof is identical to the original proof in \cite{OL-book}.
It uses Lemma \ref{lem:19.2 either q1 is small, or q1,q2 are compatible} %(Lemma 19.2 in \cite{OL-book}) 
and Items \ref{Item:period of rank i simple in any lower rank} and \ref{Item:conjugate periods are equiv}  of Lemma \ref{lem:our lem 19.3}. %(Lemma 19.3 in \cite{OL-book}). 
It also uses Item (S2) of Definition \ref{def:simple}.
\end{proof}

Lastly, we include the following lemma from Section §13 in \cite{OL-book}.%, although the lemma and its proof are identical to those in \cite{OL-book}. The reason for this seemingly repetition is that we wish to note that the statement holds in our case as is, although our notions of simple elements and periods do not fully coincide with the original definitions of \cite{OL-book}. 
\begin{lemma}[{\cite[Lemma 13.2]{OL-book}}]\label{lem:13.2 boundary of two cells can't be compatible}
Suppose that, in a reduced diagram $\Delta$ the contour of a cell $\Pi_1$ has label $A^{\pm1}$ for some period $A$ of rank $j$. Then: (1) if $\Gamma$ is a subdiagram with contour $p_1q_1p_2q_2$, where $q_1$ is a subpath in $\partial \Pi_1$ and $q_2$ is a subpath in $\partial \Pi_2$ for some cell $\Pi_2$ with contour label $A^{\pm 1}$ such that $\Pi_1$ and $\Pi_2$ are not contained in $\Gamma$, then $q_1$ and $q_2$ cannot be $A$-compatible in $\Gamma$; (2) if $\Gamma$ is a subdiagram with contour $q_1q$, where $q_1$ is a subpath in $\partial \Pi_1$ and $\Pi_1$ is not contained in $\Gamma$, then no cell in $\Gamma$ with contour label $A^{\pm 1}$ is $A$-compatible with $q_1$.\end{lemma}

\begin{proof}
The proof is identical to the original proof in \cite{OL-book}, as it only uses the fact that for each $i\geq 1$, certain simple words called periods of rank $i$ are chosen such that if $A$ is a period of rank $i$ and $A^n$ is a relator in $\calR$, then $A^n$ is a relator of rank $i$ and $A^m\notin \calR$ for all $m\neq n$. 
%This is actually a copied proof, we write it here only as a verification to the fact that it works for any 'periods' and 'periodic words', as we define it, and not as OL.
%\gil{Add an explanation here: what property of a 'period' is required here. Seems like we only use here the fact that periods are simple in rank 0. Namely, visually-simple.}
%(1) If $\Pi_1=\Pi_2$ the assertion follows from the orientability of $\Delta$, because $A^n$ does not contain two subwords $X$ and $X^{-1}$ with $|X|\geq |A|$. So let $\Pi_1\neq \Pi_2$ and let $o_1,o_2$ be the vertices on $q_1,q_2$ featuring in the definition of compatibility, that is, defining $A$-compatible decompositions of the labels of the contours $p_1$ and $p_2^{-1}$ of $\Pi_1$ and $\Pi_2$. Assuming that $p_1,p_2$ begin at $o_1,o_2$ then $\phi(p_1)=\phi(p_2)^{-1}$ in rank $0$, since $|q_i|\geq |A|$, $i=1,2$. Thus. $\Pi_1$ and $\Pi_2$ comprise a $j$-pair in $\Delta$, contrary to the fact that $\Delta$ is reduced. 
%2) If $\Gamma$ had a cell $\Pi_2$ $A$-compatible with $q_1$, then $\Pi_1$ and $\Pi_2$ would form a $j$-pair in $\Delta$, as above. \gil{Need to rephrase the use of j-pairs, and make clear the use of orientation}
\end{proof}

\end{document}